\documentclass[11pt,a4paper]{article}
\usepackage[english]{babel}
 
\usepackage{amsmath,amssymb,amsthm,epsfig,color,graphicx}
\usepackage[latin1]{inputenc}
\usepackage{color}
\usepackage{enumerate}
\newtheorem{theorem}{Theorem}
\newtheorem{proposition}[theorem]{Proposition}
\newtheorem{lemma}[theorem]{Lemma}
\newtheorem{definition}{Definition}
\newtheorem{corollary}{Corollary}

\definecolor{Red}{cmyk}{0,1,1,0}

\definecolor{Blue}{cmyk}{1,1,0,0}

\newcommand{\ba}{\begin{array}}
\newcommand{\ea}{\end{array}}
\newcommand{\ee}{\end{equation*}}
\newcommand{\ben}{\begin{enumerate}}
\newcommand{\een}{\end{enumerate}}

\let\o=\omega

\let\s=\sigma

\let\O=\Omega

\newcommand{\R}{\mathbb{R}}

\newcommand{\Z}{\mathbb{Z}}

\newcommand{\N}{\mathbb{N}}

\begin{document}

\title{A Large Deviation Principle for Gibbs States on Countable Markov Shifts at Zero Temperature}
\author{Rodrigo Bissacot\thanks{Supported by CNPq grant 312112/2015-7 and FAPESP grant 11/16265-8.
Supported also by EU Marie-Curie IRSES Brazilian-European partnership in Dynamical Systems (FP7-PEOPLE-2012-IRSES 318999 BREUDS).} \\
\footnotesize{Institut of Mathematics and Statistics (IME-USP)}\\
\footnotesize{University of S\~ao Paulo}\\
\footnotesize{\texttt{rodrigo.bissacot@gmail.com}}
\\[0.3cm]
Jairo K. Mengue\\
\footnotesize{Instituto de Matem\'{a}tica e Estat\'{i}stica}\\
\footnotesize{Universidade Federal do Rio Grande do Sul (UFRGS)}\\
\footnotesize{\texttt{jairo.mengue@ufrgs.br}}
\\[0.3cm]
Edgardo P\'erez\thanks{Supported by FAPESP grant 12/06368-7. Supported also by EU Marie-Curie IRSES Brazilian-European partnership in Dynamical Systems (FP7-PEOPLE-2012-IRSES 318999 BREUDS).}\\
\footnotesize{D\'epartement de Math\'ematiques}\\
\footnotesize{Universit\'e de Brest, 6, rue Victor le Gorgeu, 29285 Brest, France}\\
\footnotesize{\texttt{edgardomath@gmail.com}}\\
}
\maketitle

\begin{abstract}
Let $\Sigma_{A}(\N)$ be a topologically mixing countable Markov shift with the BIP pro-perty over the alphabet $\N$ and  $f: \Sigma_{A}(\N) \rightarrow \mathbb{R}$ a potential satisfying the Walters condition with finite Gurevich pressure. Under suitable hypotheses, we prove the existence of a Large Deviation Principle 
for the family $(\mu_{\beta})_{\beta > 0}$ where each $\mu_{\beta}$ is the Gibbs measure associated to the potential $\beta f$. Our main theorem generalizes from finite to countable alphabets and also to a larger class of potentials a previous result of A. Baraviera, A. O. Lopes and P. Thieullen. 
\end{abstract}

\section{Introduction and main result}
After the seminal books of R. Bowen  \cite{Bo} and D. Ruelle  \cite{Ru}, a good amount of literature was produced by dynamicists, probabilists and mathematical physicists interested in rigorous results inspired by models in statistical mechanics. This theory is known nowadays as thermodynamic formalism. Many papers are concentrated in the one-dimensional case where the state space is a finite set $S$ and $\O = S^{\Z}$ is the configuration space. Using a standard argument introduced by Sinai \cite{Sinai}, see \cite{Daon} for the adaptation to the case of countable alphabets, we can reduce this study to the models where the configurations are in $S^{\N}$ where the machinery of Ruelle operator can be applied.

In the last two decades, this theory was extended to the non-compact setting, usually called of one-dimensional unbounded spins systems by the statistical physics community when the configuration space is $S^{\Z}$ (or $S^{\N}$), where $S=\N$ (or another non-compact metric space). Many results for finite state space spin models have today the equivalent statement in the non-compact setting. These results were obtained mainly by R. D. Mauldin and M. Urba\'nski  \cite{MU2} and by the fundamental progress made by O. Sarig  \cite{Sarig01, Sarig03, Sarig04} in the setting of countable Markov Shifts. A good survey of O. Sarig's work, who received the Brin Prize in dynamical systems in 2013 for his contributions, was recently written by Pesin, see \cite{Pesin}.

Differently from the multidimensional case, transitive one-dimensional Markov shifts have at maximum one equilibrium state  \cite{Sarig04, Sarig03} when the potential has summable variation (extended to Walters potentials {in} \cite{Daon}). This means that the only possible phase transition with respect to the number of elements of the set of equilibrium measures for this class of systems is to change from the existence to the absence of the equilibrium states when we consider different inverse of temperatures $\beta$. {O. Sarig} proved that the set of positive $\beta$'s such that we have phase transition at $\beta$ can have positive Lebesgue measure \cite{Sarig05, Sarig06}, even in the case where the potential depends on a finite number of coordinates. {G. Iommi} made a link with ergodic optimization proving that the absence of phase transitions in Renewal shifts is equivalent to the existence of maximizing measures  \cite{Iommi}.

In addition to phase transitions, we have others connections between ergodic optimization and thermodynamic formalism. The accumulation points at zero temperature of equilibrium states are maximizing measures  \cite{BLL, CLT, Jenkinson, JMU0}. Sub-actions (dual objects of maximizing measures)  can be constructed as accumulation points of $\frac{1}{\beta}\log h_\beta$, where $h_\beta$ is the main eigenfunction of the Ruelle operator associated to $\beta f$. In the proof of the Large Deviation Principle (LDP) for the family of equilibrium measures $(\mu_\beta)_{\beta}$ when the temperature goes to zero, already established on the compact setting, the deviation function is written in terms of sub-actions \cite{BLT, LM, LM1}.

The contributions of our paper are: it is presented an alternative argument to prove the LDP without the involution kernel used in the original proof \cite{BLT}. Moreover, we prove that this strategy also works in the non-compact setting where the equilibrium measures have the Gibbs property, that is, topologically mixing countable Markov shifts with the BIP property. Furthermore, we generalize some results about zero temperature limits for potentials satisfying the Walters condition: the existence of maximizing measures, sub-actions and maximizing sets.

One of the main assumptions in our proofs is the uniqueness of the
maximizing measure. This hyphotesis is also used to guarantee that the deviation function $I$ is well defined, see equation $(\ref{deviation function})$ below. Additionally, in \cite{BLM} it is presented a class of potentials which have more than one maximizing measure; the authors show that a Large Deviation Principle (if it exists) does not have the same deviation function $I$. In \cite{Mengue2} it is exhibited the deviation function for this potentials. The uniqueness is a generic condition in several spaces of potentials in the compact case \cite{BLL, Bou, CLT, Jenkinson} and dense in suitable spaces defined in the non-compact setting \cite{DUZ}. 

Despite the fact that Gibbs and equilibrium measures are central objects in thermodynamic formalism and that the study of ground states is one of the main questions when we analyze a model in statistical mechanics, concerning Large Deviation Principles at zero temperature we are aware of only the few papers cited above. 

Our main result is the following:

\begin{theorem}\label{teoremaA}
Let $\Sigma_{A}(\N)$ be a topologically mixing countable Markov shift with the BIP property over the alphabet $\N$. Let  $f: \Sigma_{A}(\N) \rightarrow \mathbb{R}$ be a potential satisfying the Walters condition with $\operatorname{Var}_1(f)<\infty$ and finite Gurevich pressure.  Also assume that $f$ admits a unique maximizing measure  $\mu$. For each $\beta>1$,
denote $\mu_\beta$ the unique equilibrium measure associated to $\beta f$. Then $(\mu_\beta)_{\beta}$  satisfies
a Large Deviation Principle, that is, for any cylinder $C$ of $\Sigma_A(\mathbb{N})$,
\begin{equation*}
\lim_{\beta\rightarrow \infty} \frac{1}{\beta}\log\mu_{\beta}(C)=-\inf_{x\in C}I(x)
\end{equation*}
where
\begin{equation}\label{deviation function}
I(x)=-\sum_{j=0}^\infty (f + V - V\circ \sigma-\mu(f))(\sigma^j(x))
\end{equation}
and $V$ is any bounded calibrated sub-action for $f$.
\end{theorem}

As a corollary, we obtain a new proof of the same principle in the case of topologically mixing subshifts over a finite alphabet previously proved by A. Baraviera, A. Lopes and P. Thieullen  \cite{BLT}.

The paper is organized as follows. In section 2 we fix the setting, we review some results of the thermodynamic formalism and ergodic optimization on countable Markov shifts. In section 3 we extend some results about zero temperature to potentials sa-tisfying the Walters condition in the non-compact setting. We also prove the existence of maximizing measures, calibrated sub-actions and maximizing sets. In section 4, we prove the main result, that is, the Large Deviation Principle at zero temperature for the family of Gibbs measures  $(\mu_\beta)_{\beta}$. Finally,  in section 5, we point out some remarks and further directions to be investigated.

We should remark that the results of this paper are substantially contained in the Ph.D. thesis of two of us. The {main ideas to an} alternative proof of the Large Deviation Principle on the compact setting was done in \cite{Mengue},  and the adaptation to the countable Markov shifts was obtained in \cite{Pe}.

\section{Thermodynamic formalism and ergodic optimization on countable Markov shifts}


Let $\mathbb{N}$ be the set of non-negative integers and $\Sigma(\mathbb{N})$ be the set of sequences of elements in
$\mathbb{N}$. We will always assume that there exists an infinite matrix $A:\mathbb{N}\times \mathbb{N} \rightarrow \{0,1\}$, such that $\Sigma_A(\mathbb{N}):=\{x\in \Sigma(\mathbb{N}): A(x_i,x_{i+1})=1,  \ \forall i\geq 0\}$, that is, our subshift $\Sigma_A(\mathbb{N})$ is \textit{Markov}. As usual,  the subsets $[a_0,\ldots,a_{n-1}]=\{x\in \Sigma_A(\mathbb{N}): x_i=a_i \ \text{when} \ 0\leq i \leq n-1\}$ are called \textit{cylinders}.
A \textit{word} $w$ is a finite concatenation of symbols in $\mathbb{N}$ and we say that a word is called $A$-\textit{admissible }or simply \textit{admissible} if the cylinder $[w]$ is non-empty. In particular, given a symbol $a \in \mathbb{N}$, we say that $a$ is admissible when $[a] \neq \emptyset$.

We say that $\Sigma_A(\mathbb{N})$ is \textit{finitely primitive} when there exist a $K_0$ in $\mathbb{N}$ and a finite subalphabet $\mathbb{F} \subset \mathbb{N}$ such that for any pair of admissible symbols $a,b$ in $\mathbb{N}$ there are  $i_1,...,i_{K_0}$ in $\mathbb{F}$ satisfying
 $A(a,i_1)\cdot A(i_1,i_2)\cdot\cdot\cdot A(i_{K_0-2},i_{K_0-1})\cdot A(i_{K_0},b)=1$. This condition says that for any pair of admissible symbols in the alphabet $\mathbb{N}$ can be connected by an admissible word of  $K_0$ symbols $w = i_1 i_2 ... i_{K_0-1} i_{K_0}$ where all the symbols $i_j$ belong to the finite alphabet $\mathbb{F}$.

Our dynamics is given by $\sigma:\Sigma_A(\mathbb{N})\rightarrow \Sigma_A(\mathbb{N})$, 
$\sigma(x_0,x_1,x_2,\ldots)=(x_1,x_2,\ldots)$, the standard \textit{shift map}.
We will consider $\Sigma_A(\mathbb{N})$ as a metric space. Fixed $r\in (0,1)$,  for any two sequences $x,y$ in $\Sigma_A(\mathbb{N})$ the distance is defined by $d(x,y)=r^{t(x,y)}$, where $t(x,y)=\inf(\{k:x_k \neq y_k\} \cup \{\infty\})$. With this metric $\Sigma_A(\mathbb{N})$ is a Polish space, the metric generates the product topology which has the cylinders as a base and the shift map $\sigma$ is continuous with respect to this topology. We denote $\mathcal{M}_\sigma(\Sigma_A(\mathbb{N}))$ the set of $\sigma$-invariant Borel probability measures.

Unless we say the contrary, $\Sigma_A(\mathbb{N})$ will be a \textit{topologically mixing} Markov shift: for each pair of admissible symbols $a,b$ there exist a natural number $N(a,b)$ such that for any $n\geq N(a,b)$ we have $[a] \cap \sigma^{-n}([b]) \neq \emptyset$. A good part of the literature uses the notion of \textit{BIP (big images and preimages property)}: there exists a finite set $B \subset \mathbb{N}$ such that for any admissible $a$ in $\mathbb{N}$ we can find $b_1,b_2$ in $B$ such that the word $b_1ab_2$ is admissible. When $\Sigma_A(\mathbb{N})$ is topologically mixing the notions of BIP and finitely primitive coincide and we will use both indiscriminately.

We equip $\mathcal{M}_\sigma(\Sigma_A(\mathbb{N}))$ with the weak* topology. In this way
 $(\mu_n)_{n\geq 1}$ converges weakly* to $\mu$ if and only if
$\int g d \mu_n \rightarrow \int g d\mu$ for every bounded continuous function $g:\Sigma_A(\mathbb{N})\rightarrow \mathbb{R}$. This is equivalent to $\mu_n(C) \rightarrow \mu(C)$ for all cylinder {set} $C$. We denote $C^0(\Sigma_A(\mathbb{N}))$ by the space of continuous real-valued functions on $\Sigma_A(\mathbb{N})$, equipped with the topology of uniform convergence on compact subsets.
 
 Given a function $f:\Sigma_A(\mathbb{N})\to\mathbb{R} $, we define  $S_nf=\sum\limits_{j=0}^{n-1}f \circ \sigma^j$ and $S_0f=0$.
The \textit{$n$-th variation of
$f$} is the number
$\operatorname{Var}_nf:=\sup\{|f(x)-f(y)|; x_0=y_0,\ldots,x_{n-1}=y_{n-1}\}$
and we say that $f$ has \textit{summable variations} when $\operatorname{Var}(f) := \sum_{n=1}^{\infty} \operatorname{Var}_n f < \infty.$ When, in addition, there exist a $K>0$ such that for all $n \in \mathbb{N}$ we have $\operatorname{Var}_nf \leq Kr^{n}$, then we say that $f$ is locally H\"older.

\begin{definition}
A function $f:\Sigma_A(\mathbb{N}) \rightarrow \mathbb{R}$ satisfies the
Walters condition when for every $k\geq 1$ we have \ $\sup\limits_{n\geq 1}[\operatorname{Var}_{n+k}S_nf]<\infty$ and $\lim\limits_{k \rightarrow \infty}\sup\limits_{n\geq 1}[\operatorname{Var}_{n+k}S_nf]=0.$
\end{definition}

The locally H\"older condition clearly implies summable variation, the last one the usual regularity in papers nowadays about thermodynamic formalism on countable Markov shifts. These two conditions are stronger and imply the Walters regularity. In any of the three cases the function will be uniformly continuous, but it is not necessarily bounded since the space is non-compact.

{Given a function $f:\Sigma_A(\mathbb{N})\to\mathbb{R}$ satisfying the Walters condition we define the \textit{Gurevich pressure  of $f$} by} $P(f) := \lim_{n\to\infty}\frac{1}{n}\log \sum_{\sigma^{n}(x)=x}e^{S_n(f(x))}\chi_{[a](x)}.$ The definition does not depend of the symbol  $a$ (see \cite{Sarig03}, Proposition 3.2).

Regarding the regularity of the potential, one of the most general assumptions today where it is known that the machinery of the Ruelle operator can be applied (see \cite{Daon, Sarig03}) in the thermodynamic formalism is:

\medskip

\textbf{Assumption I:} The potential $f:\Sigma_A(\mathbb{N}) \rightarrow \mathbb{R}$ is Walters with $\operatorname{Var}_1(f)<\infty$.

\medskip

In the BIP case, finite Gurevich pressure and Assumption I imply that:
 \begin{equation}
 \displaystyle\sum_{i\in \mathbb{N}}\exp(\sup f|_{[i]})<\infty
 \end{equation}
 In particular, the potential is \textit{coercive}, that is, $\displaystyle\lim_{i \rightarrow \infty}(\sup f|_{[i]}) = - \infty$.
 
 For the class of potentials satisfying this assumption we have the \textit{variational principle} for the pressure:

\begin{theorem} [Sarig]\label{teoremaSarig} Let $\Sigma_{A}(\N)$ be a topologically mixing countable Markov shift. Let  $f: \Sigma_{A}(\N) \rightarrow \mathbb{R}$ be a potential satisfying the Assumption I, with finite Gurevich pressure and $\sup f < \infty$. Then,
\begin{equation}\label{variational}
P(f) = \sup\left\{ h(\mu) + \int f\, d\mu\,|\, \mu \in \mathcal{M}_{\sigma}(\Sigma_A(\mathbb{N}))\,\, \textsl{and} \,\,\, \int f\, d\mu >-\infty\right\}.
 \end{equation}
\end{theorem}
\noindent

The invariant Borel probability measures which attain the supremum (\ref{variational}) are called \textit{equilibrium measures}. Since the space $\Sigma_{A}(\N)$ is not compact, the existence of equilibrium measures is not a trivial problem and depends of a combination of properties of the infinite matrix $A$ which defines the shift $\Sigma_{A}(\N)$ and the potential $f$. 

From the dynamical point of view, the minimal assumption considered on the shift $\Sigma_{A}(\N)$ is to assume that the matrix $A$ is  \textit{irreducible}: for each pair of admissible symbols $a$ and $b$ in $\N$, there exist an admissible word $\o_{ab}$ which connects $a$ and $b$. In other words, $a\o_{ab} b$ is admissible. The matrix $A$ be irreducible is equivalent to the shift map $\s$ be \textit{transitive}. And each transitive countable Markov shift $\Sigma_{A}(\N)$ can be decomposed in a finite number of topological mixing Markov shifts and the equilibrium measures

$$L_f:C_b(\Sigma_A(\mathbb{N}))\rightarrow C_b(\Sigma_A(\mathbb{N})), \hspace{0.3cm}
(L_f g)(x):=\sum\limits_{\sigma(y)=x}e^{f(y)}g(y).$$
In this case $\displaystyle{\|L_f\|_0\leq \sum_{i\in \mathbb{N}}\exp{(\sup f|_{[i]})}}$. 

\begin{theorem}[Generalized Ruelle's Perron-Frobenius Theorem \cite{Daon,Sarig03}]\label{ruelle}
Suppose $f$ satisfies the Assumption I. There are a $\lambda>0$, a positive
continuous function $h$, and a conservative measure $\nu$ which is finite on cylinders,
such that $L_{f}h=\lambda h$, $L^*_{f}\nu=\lambda \nu$, and $\int hd\nu=1$. In this
case $\lambda=e^{P(f)}$ {and for every $a\in\mathbb{N}$,  $\lambda^{-n}(L^n_f \chi_{[a]})\rightarrow h \cdot \nu([a])$
uniformly on compacts.}
\end{theorem}

\noindent
{For a prove of the theorems \ref{teoremaSarig} and \ref{ruelle}, see \cite{Daon} and \cite{Sarig03}.
 
{We call an invariant measure $\mu$ supported {on the} shift space $\Sigma_A(\mathbb{N})$ a
\textit{Gibbs measure for $f$} (in the sense of Bowen) if there exist 
constants $C_1,C_2>0$ such that

\begin{equation}\label{Gibbseq}
C_1\leq \frac{\mu[x_0\ldots x_{n-1}]}{\exp(S_n f(x)-nP(f))}\leq C_2
\end{equation}
for all $x\in \Sigma_A(\mathbb{N})$.}

In the zero temperature case is considered the family of functions $\beta f$ and the possible limits of the above objects when $\beta \to \infty$. The thermodynamic interpretation of the parameter $\beta$ is as the inverse temperature of the system.

A  $\sigma$-invariant probability $\mu$ is said to be maximizing to $f$ if
$$m(f):=\int fd\mu=\sup_{\nu \in \mathcal{M}_{\sigma}(\Sigma_A(\mathbb{N}))}\int f d\nu.$$
 After {Jenkinson's lecture notes}
 \cite{Jenkinson}
the study of maximizing measures became known as
Ergodic Optimization.  
The main conjecture in the
area, until recently, was that for an expanding 
transformation the maximizing measure of a generic
H\"older or Lipschitz {function} is supported on a single periodic
orbit. After partial results \cite{CLT,AQuas}, 
the conjecture was proved by G. Contreras \cite{GC}. 

We will show that {$f$ has at least one maximizing measure} and any accumulation point of $\mu_\beta=\mu_{\beta f}$ when $\beta\to +\infty$ is maximizing to $f$ ({see \cite{JMU0,Morris} for  proofs} when $f$ has summable variation). 
In \cite{BG} is proved, for a locally H\"older continuous and coercive potential on 
 a finitely primitive shift, the existence of maximizing measure, moreover,
 their support lies in a particular Markov subshift on a finite
 alphabet. Recently, {
 	in} \cite{BF} (see Theorem $\ref{Bissacot}$) {is} proved that for a coercive potential with summable variation  over a topologically transitive shift there exists a maximizing 
 measure whose support lies on a subshift on a finite alphabet.  

The existence of the above limit  is not a
 straightforward problem. In \cite{Bremont} is proved the existence of the limit $\mu_\infty$ for locally constant potentials on a finite topologically mixing
Markov shift. This results were extended in \cite{Leplaideur} to equilibrium states
 $\mu_{\beta f+g}$, where $f$ is locally constant and g is H\"older continuous. However in \cite{Chazottes} is shown an example of H\"older function $f$ where $\mu_{\beta f}$ does not converge. In \cite{Kempton} is proved the convergence of $\mu_{\beta f}$ for a locally constant potential $f$ on a countable shift with the BIP property.

\vspace{0.5cm}


A \textit{sub-action} (for the potential $f$) is a  function $V\in C^0(\Sigma_A(\mathbb{N}))$ satisfying
\begin{equation}\label{ec01}
f(x)+V(x)-V(\sigma(x)) -m(f)\leq 0, \ \ \forall  x\in \Sigma_A(\mathbb{N}).
\end{equation}
We say that the sub-action is \textit{calibrated} if for each $y\in \Sigma_A(\mathbb{N})$ there exists $x\in \sigma^{-1}(y)$ verifying the equality in (\ref{ec01}).
We will show that, in the zero temperature limit, the eigenfunctions $h_\beta = h_{\beta f}$ for the Ruelle operator $L_{\beta f}$ will correspond to calibrated  sub-actions,
that is, any accumulation function of the family $\frac{1}{\beta}\log(h_\beta)$, with respect to the topology of convergence in compact sets, is a bounded calibrated sub-action.

We say a family of probabilities $\mu_\beta$, $\beta>0$, satisfies a \textit{Large Deviation Principle} if there exists a function $I:\Sigma_A \rightarrow \mathbb{R}\cup \infty$, (the deviation function) which is non-negative, lower
semi-continuous, and such that for any cylinder $C$, 
$$\lim_{\beta \rightarrow \infty}\frac{1}{\beta}\log \mu_\beta(C)=-\inf_{x\in C}I(x).$$

If the maximizing measure is unique then we will show that any two bounded calibrated sub-actions differ by a constant, this in particular show us that the rate function $I$ \text{in} Theorem \ref{teoremaA} is well defined. 

\section{Ergodic Optimization and the zero temperature limit}

In all this section  we suppose that $\Sigma_A(\mathbb{N})$ is a sub-shift finitely primitive and $f:\Sigma_A(\mathbb{N})\to\mathbb{R}$  satisfies the Walters condition, with $\operatorname{Var}_1(f)<\infty$. and $P(f)<\infty$.
In this case, for each $\beta >1 $ we have $P(\beta f)<\infty$. Following \cite{Sarig03}, from our hypothesis, there exists a constant $C_\beta>0$ such that $ C_{\beta}^{-1}<h_\beta<C_\beta$, where  $h_\beta$ is the eigenfuntion for the Ruelle Operator $L_{\beta f}$ associated with the eigenvalue $e^{P(\beta f)}$. In this way we can assume that the eigenmeasure $\nu_{\beta}$ is a probability and the eigenfunction is normalized from $\int h_\beta\, d\nu_\beta =1$. The probability $\mu_\beta$ defined from $d\mu_\beta = h_\beta d\nu_\beta$ is  the $\sigma-$invariant Gibbs measure of $\beta f$. As we will see below, if $\beta>1$ then $\mu_\beta$ is also the equilibrium measure of $\beta f$.

{The following lemma appears as a hypothesis in \cite{JMU0}. Here it is a consequence of our hypothesis.}
 
\begin{lemma}\label{lemmasumabble} 
 $$\sum_{i\in \mathbb{N}}\exp(\sup f|_{[i]})<\infty.$$
 Particularly, $f$ is bounded above and coercive, that is, $\displaystyle{\lim_{i\to\infty} \sup f|_{[i]} = -\infty}.$

\end{lemma}

\begin{proof} 

For each $a\in \mathbb{N}$ and $n\geq K_0$ there exist $x_1,x_2,\ldots,x_n \in \mathbb{F}$
such that the periodic sequence $\overline{ax_1\cdots x_n}$ with period $ax_1x_2x_3\ldots x_n$ is admissible. For each fixed $x_1,\ldots,x_n \in \mathbb{F}$ denote by $X(x_1,\ldots,x_n)$ the set of symbols $a\in \mathbb{N}$ such that $\overline{ax_1\cdots x_n}$ is admissible.
Using that $P(f)$ is finite, {for each fixed $x_1 \in \mathbb{F}$ there exists a natural number $n$ large enough such that$$\sum_{\sigma^{n+1}(x)=x}\exp(S_{n+1}f(x))\chi_{[x_1]}(x)<\infty.$$
As $\mathbb{F}$ is a finite set we can fix a large $n$ independent of the choice of $x_1$.}

This implies that $$\sum_{a\in X(x_1,\ldots,x_n)} \exp(S_{n+1}f(\overline{ax_1\cdots x_n}))<\infty$$ for each $x_1,...,x_n \in \mathbb{F}$.
Thus
\begin{equation*}\label{ec62}
\sum_{a\in X(x_1,\ldots,x_n)}\exp\big[ f(\overline{ax_1\ldots x_n})+f(\overline{x_1\cdots  x_n a})+\cdots
+f(\overline{x_nax_1\cdots x_{n-1}})\big]<\infty.
\end{equation*}
As $ x_1,\ldots,x_n \in \mathbb{F}$, $f$ satisfies the Walters condition and $\operatorname{Var}_1(f)<\infty$, then $$f(\overline{x_1\cdots x_n a})+\cdots\\
+f(\overline{x_nax_1\cdots x_{n-1}})$$ is bounded by a constant that does not depend of $a,x_1,\ldots,x_n$.

Therefore 
$$\sum_{a\in X(x_1,\ldots,x_n)}\exp(f(\overline{ax_1\cdots x_n}))<\infty,$$
and multiplying by $e^{\operatorname{Var}_1(f)}$ we get
$$\sum_{a\in X(x_1,\ldots,x_n)}\exp(\sup f|_{[a]})<\infty.$$
Each $a\in \mathbb{N}$ appears in at least one of the sets $ X(x_1,\ldots,x_n)$
therefore summing over the finite set of sequences $x_1,\ldots,x_n\in \mathbb{F}$ we get
$$\sum_{a\in \mathbb{N}}\exp(\sup f|_{[a]})<\infty.$$
\end{proof}



Now we study the family $\frac{1}{\beta}\log(h_\beta)$. The initial objective is to show that can be applied the Arzel\`{a}-Ascoli theorem for this family. We will prove that the accumulation points of this family as $\beta \to\infty$  are calibrated sub-actions for $f$. We start with the following lemma:

\begin{lemma}\label{propo1}
Given $k\geq 1$, for any $x,y \in \Sigma_A(\mathbb{N})$
such that $d(x,y)\leq r^{k}$  we have
$$h_\beta(x)\leq h_\beta(y)e^{\beta M_k},$$
where $M_k:=\sup_{n\geq 1}\operatorname{Var}_{n+k}S_nf$.
\end{lemma}

\begin{proof}
We will prove that $h_\beta(x)- h_\beta(y)e^{\beta M_k}\leq 0.$
Given  $x\in \Sigma_A(\mathbb{N})$ we define
 $$P^n(x)=\{a=(a_0,\ldots,a_{n-1}):ax \in \Sigma_A(\mathbb{N})\}.$$
By hypothesis $d(x,y)<1$, and so $P^n(x)=P^{n}(y)$.

Since  $\lim\limits_{n\rightarrow \infty}\frac{1}{\lambda_\beta^n}L^n_{\beta f}\chi_{[a]}(x)=h_\beta(x)\nu_\beta([a])$
uniformly in compacts (see Theorem \ref{ruelle}), we have

\begin{align*}
(h_\beta(x)- h_\beta(y)e^{\beta M_k})\nu_\beta([a])&=
\lim_{n\rightarrow \infty} \frac{1}{\lambda_\beta^n}\big(L^n_{\beta f}\chi_{[a]}(x)-L^n_{\beta f}\chi_{[a]}(y)e^{\beta M_k} \big)\\
&=\lim_{n\rightarrow \infty} \frac{1}{\lambda_\beta^n} \bigg(\sum_{{p}\in [a]\cap P^n(x)} e^{\beta S_nf({p}x)}-e^{\beta S_nf({p}y)+\beta M_k}\bigg)\\
&=\lim_{n\rightarrow \infty} \frac{1}{\lambda_\beta^n} \sum_{{p}\in [a]\cap P^n(x)} e^{\beta S_nf({p}y)}\big[e^{\beta S_nf({p}x)-\beta S_nf({p}y)}-e^{\beta M_k}\big].
\end{align*}

Finally, we observe that for any $n\geq 0$,
$$
e^{\beta S_nf({p}x)-\beta S_nf({p}y)}-e^{\beta M_k}\leq 0 ,
$$
because by definition of $M_k$,
$$
S_nf({p}x)-S_nf({p}y)\leq \operatorname{Var}_{n+k}S_nf\leq  M_k.
$$
Thus, we conclude the proof.

\end{proof}

\begin{corollary}\label{variacaoh}
 For each $k$, $\operatorname{Var}_k(\frac{1}{\beta}\log(h_\beta)) $ is bounded by a constant $M_k$ that does not depend of $\beta$, such that $\lim\limits_{k\to\infty}M_k = 0$ .
\end{corollary}
\begin{proof}
It is a consequence of the above lemma and of the Walters condition.
\end{proof}

\begin{lemma}\label{propo2}  There exist constants $C_1$ and $C_2$ such that for $\beta\geq 1$
\begin{equation*}
e^{C_1 \beta}\leq h_\beta \leq e^{C_2 \beta}.
\end{equation*}
\end{lemma}

\begin{proof}
Fix $x\in \Sigma_A(\mathbb{N})$ and a point $y=(y_0,y_1,\ldots)\in \Sigma_A(\mathbb{N})$ satisfying $h_\beta(y)\geq 1/2$. Using that $\Sigma_A(\mathbb{N})$ is finitely primitive
we know that there exist $w_1,\ldots,w_{K_0}\in \mathbb{F}$ such that $y_0w_1\ldots w_{K_0}x_0$ is admissible. Let $z=y_0w_1\ldots w_{K_0}x$.
Then $d(y,z)<1$ and from Lemma $\ref{propo1}$
we have
\begin{equation*}\label{cap3-ec10}
\frac{1/2}{h_\beta(z)}\leq\frac{h_\beta(y)}{h_\beta(z)}\leq e^{\beta M_1}.
\end{equation*}
Therefore, if $\beta\geq 1$,
\begin{align*}
h_\beta(x)=\frac{(L^{K_0+1}_{\beta f}h_\beta)(x)}{\lambda^{K_0+1}_\beta}&\geq e^{\beta S_{K_0+1}(f)(z)-(K_0+1)P(\beta f)}h_\beta(z)\\
&\geq e^{\beta S_{K_0+1}(f)(z)-(K_0+1)\beta P(f)} \frac{e^{-\beta M_1}}{2} \\
&\geq e^{\beta S_{K_0+1}(f)(z)+\beta C},\\
\end{align*}
where $C=-(K_0+1)P(f) -M_1-\log(2)$ is a constant.

As $f$ satisfies the Walters condition and $\operatorname{Var}_1(f)<\infty$ we can find $x' \in \mathbb{F}^{\mathbb{N}}$  such that
\begin{align*}
S_{K_0+1}(f)(z)&=f(y_0w_0\ldots w_{K_0-1}x)+f(w_0\ldots w_{K_0-1}x)+\cdots +f(w_{K_0-1}x) \\
&\geq f(y_0w_0\ldots w_{K_0-1}x')+\cdots+
f(w_{K_0-1}x')-\sup_{k\geq 1}\operatorname{Var}_{k+1}S_kf-\operatorname{Var}_1(f)\\
&\geq (K_0+1)m-\sup_{k\geq 1}\operatorname{Var}_{k+1}S_kf-\operatorname{Var}_1(f) =:\tilde{C},
\end{align*}
where
$$m:=\min\{f(w)\,| w \in [\{y_0\}\cup \mathbb{F}]^{\mathbb{N}}\}> -\infty.$$
We conclude that
\begin{equation*}
e^{\beta (C+\tilde{C})}<h_\beta(x),
\end{equation*}

\bigskip

Now we prove the other inequality. Let $y$ be a point satisfying $h_\beta(y)\leq 2$ and fix $a\in \mathbb{F}$.  From the finitely primitive assumption, there exist $w_1,\ldots, w_{K_0}\in \mathbb{F}$ such that $aw_1\ldots w_{K_0}y_0$  is admissible.
 Fix $x\in[a]$ and define
 $z':=aw_1\ldots w_{K_0}y$. Thus $d(z',x)<1$  and from Lemma  $\ref{propo1}$
we get
\begin{equation*}\label{cap3-ec11}
\frac{h_\beta(x)}{h_\beta(z')}\leq e^{\beta M_1}.
\end{equation*}
Then
\begin{align*}
h_\beta(y)&=\frac{(L^{K_0+1}_{\beta f}h_\beta)(y)}{\lambda^{K_0+1}_\beta}\geq e^{\beta S_{K_0+1}(f)(z')-(K_0+1) P(\beta f)}h_\beta(z')\\
&\geq e^{\beta S_{K_0+1}(f)(z')-(K_0+1)\beta P(f)} h_\beta(x)e^{-\beta M_1}.\\
\end{align*}
Therefore
\[h_\beta(x) \leq 2e^{ -\beta S_{K_0+1}(f)(z')+(K_0+1)\beta P(f) +\beta M_1}.\]
As $a, w_1,\ldots,w_{K_0}\in \mathbb{F}$, following computations as above, for $\beta\geq 1$,

$$h_\beta(x)< e^{\beta \bar{C_2}},$$
for some constant $\bar{C_2}$ depending of $a$ but not of $x$. 
As $\mathbb{F}$ is finite we can suppose this inequality true for any $a\in \mathbb{F}$ and any $x\in[a]$.


Fix a point $z^a\in [a]$, for each $a\in \mathbb{F}$. Given any $x\in \Sigma_A$, for every point of the form
$(\xi_0,\ldots,\xi_{n-1},x)$ we can find some $a \in \mathbb{F}$ such that
$(\xi_0,\ldots,\xi_{n-1},z^a) \in \Sigma_A$.
As $f$ satisfies the Walters condition and $\operatorname{Var}_1(f)<\infty$,
\begin{align*}
&S_nf(\xi_0,\ldots,\xi_{n-1},x)-S_nf(\xi_0,\ldots,\xi_{n-1},z^a)\\
&=(S_{n-1}f(\xi_0,\ldots,\xi_{n-1},x)-S_{n-1}f(\xi_0,\ldots,\xi_{n-1},z^a))+(f(\xi_{n-1}x)-f(\xi_{n-1}z^a))\\
&\leq \sup_{k\geq 1}\operatorname{Var}_{k+1}S_k f+\operatorname{Var}_1(f)=:M.
\end{align*}

Then
$$e^{\beta S_nf(\xi_0,\ldots,\xi_{n-1},x)}\leq e^{\beta M}e^{\beta S_nf(\xi_0,\ldots,\xi_{n-1},z^a)}$$
and summing over all possible $\xi$ and $a's$ we get {for any fixed $b\in\mathbb{N}$} 
$$\lambda^{-n}_{\beta}(L^n_{\beta f} \chi_{[b]})(x)\leq \sum_{a\in F} e^{\beta M}\lambda^{-n}_{\beta}(L^n_{\beta f} \chi_{[b]})(z^a).$$
Taking the limit when $n\rightarrow \infty$ we obtain
\begin{equation*}
h_\beta(x)\nu_\beta([b])\leq \sum_{a\in F} e^{\beta M}h_\beta(z^a)\nu_\beta([b]).
\end{equation*}
Using that $h_\beta(z^a)\leq e^{\beta \bar{C_2}} \ \text{for every} \ a\in \mathbb{F}$, we conclude the existence of a constant $C_2$ such that for $\beta\geq 1$
$$h_\beta(x)\leq e^{\beta C_2}.$$
\end{proof}

\begin{corollary}\label{equicontinuous}
 The sequence $V_\beta:=\frac{1}{\beta}\log h_\beta$ is equicontinuous and
uniformly bounded. Particularly, $V_\beta$ has a subsequence that converges uniformly in each compact set of $\Sigma_A(\mathbb{N})$. 
\end{corollary}
\begin{proof} It is a consequence of Corollary \ref{variacaoh}, Lemma \ref{propo2} and the Arzel\`{a}-Ascoli theorem to separable spaces.
\end{proof}

\begin{Remark}
In the absence of the BIP assumption the eigenfunction $h_\beta$
could be closer to zero and/or the eigenmesure $\nu_\beta(\Sigma_A(\N))=\infty$
(see example 2 in \cite{Sarig01}).
\end{Remark}

\bigskip

Now, we prove the main result of this section: 

\begin{proposition} \label{propo-subaction}
 Any possible limit function
  $V(x):=\lim\limits_{\beta_i \rightarrow \infty}\frac{1}{\beta_i}\log h_{\beta_i}(x)$ is a bounded calibrated sub-action for the potential $f$.
\end{proposition}

\begin{proof}

There exists a subsequence such that $\frac{1}{\beta_n}\log h_{\beta_n}$ converges uniformly
in every compact set $K\subset \Sigma_A(\mathbb{N})$. 
%

Abusing of the notation we write
$$V(x):=\lim_{\beta \rightarrow \infty}\frac{1}{\beta}\log h_{\beta}(x).
$$
Then $h_\beta(z)=e^{\beta(V(z)+ \epsilon_\beta(z))}$, with
$\lim\limits_{\beta\rightarrow \infty}|\epsilon_\beta|=0$ uniformly in compacts.
Also, as there exists a constant $C$ such that for $\beta \geq 1$
\[e^{-\beta C} \leq h_\beta \leq e^{\beta C}\] we conclude that $-C\leq V\leq C$ and consequently $-2C\leq \epsilon_\beta \leq 2C$ for any $\beta\geq 1$.

We want to show that $V$ is a calibrated sub-action. In this way we
fix  $x \in \Sigma_A(\mathbb{N})$. In order to simplify the arguments below we will write the sums in all $j\geq 0$ instead of to take the sum only on $j's$ such that $jx_0$ is admissible.  From Lemma \ref{lemmasumabble} there exists $j_0$ large enough such that
\begin{equation*}\label{cap3-ec6}
\sum_{j=j_0+1}^\infty e^{f(jx)}<1
\end{equation*}
and consequently $f$ is coercive. Then, as $V$ and $\epsilon_\beta$ are bounded, there exists $j_1$ large enough such
that
\begin{equation*} \label{cap3-ec7}
\sup_{j\in \mathbb{N}} (f(jx)+V(jx))= \max_{0\leq j\leq j_1-1}(f(jx)+V(jx))=:\alpha
\end{equation*}
and for any $\beta\geq 1$
\begin{equation*} \label{cap3-new}
\sup_{j\in \mathbb{N}} (f(jx)+V(jx)+\epsilon_\beta(jx))= \max_{0\leq j\leq j_1-1}(f(jx)+V(jx)+\epsilon_\beta(jx))=:\alpha_\beta.
\end{equation*}
Particularly
\[\lim_{\beta\to\infty}\alpha_\beta = \lim_{\beta\to\infty}\max_{0\leq j\leq j_1-1}(f(jx)+V(jx)+\epsilon_\beta(jx)) =\lim_{\beta\to\infty}\max_{0\leq j\leq j_1-1}(f(jx)+V(jx))=\alpha.\] 

Writing
\begin{equation*}
\sum_{j=0}^\infty e^{\beta f(jx)}h_\beta(jx)=\sum_{j=0}^{j_0} e^{\beta f(jx)}h_\beta(jx)+\sum_{j=j_0+1}^\infty e^{\beta f(jx)}h_\beta(jx)
\end{equation*}
we have
$$\sum_{j=0}^{j_0} e^{\beta f(jx)}h_\beta(jx)=\sum_{j=0}^{j_0} e^{\beta f(jx)}e^{\beta(V(jx)+ \epsilon_\beta(jx))}\leq (j_0+1)e^{\beta \alpha_\beta}$$
and
\begin{align*}
\sum_{j=j_0+1}^\infty e^{\beta f(jx)}h_\beta(jx)&=\sum_{j=j_0+1}^\infty e^{\beta\big(f(jx)+V(jx)\big)+\beta \epsilon_\beta(jx)}\\
&=\sum_{j=j_0+1}^\infty e^{f(jx)+(\beta-1)\big[f(jx)+V(jx)+\epsilon_\beta(jx)\big]+V(jx)+ \epsilon_{\beta}(jx)}\\
&\leq \bigg(\sum_{j=j_0+1}^\infty e^{f(jx)}\bigg)e^{(\beta-1)(\alpha_\beta)+C} \\
&<1\cdot e^{(\beta-1)(\alpha_\beta)+C}.
\end{align*}

Thus, we obtain
\[\sum_{j=0}^\infty e^{\beta f(jx)}h_\beta(jx)\leq (j_0+1)e^{ \beta \alpha_\beta } +  e^{(\beta-1)(\alpha_\beta)+C} \]
and so
\begin{align*}
m(f)+V(x) & = \lim_{\beta \to \infty}\frac{1}{\beta}\log(\lambda_\beta\cdot h_\beta(x))=     \lim_{\beta \to \infty}\frac{1}{\beta}\log \left(\sum_{j=0}^\infty e^{\beta f(jx)}h_\beta(jx)\right)\\
&\leq  \max\left\{\lim_{\beta\to\infty} \frac{1}{\beta}\log\left( (j_0+1)e^{\beta\alpha_\beta}\right) , \lim_{\beta\to\infty} \frac{1}{\beta}\log\left(  e^{(\beta-1)(\alpha_\beta)+C}\right)\right\}\\
&=\alpha=\sup_{j\in \mathbb{N}} (f(jx)+V(jx)).
\end{align*}

In order to prove the other inequality we observe that if $$(f+V)(i_0x)=\max\limits_{0\leq j\leq j_1-1}(f+V)(jx)=\sup_{j\in \mathbb{N}} (f(jx)+V(jx)),$$ 

then
\begin{align*}
V(x) +m(f) &= \lim_{\beta\to\infty} \frac{1}{\beta} \log\left(\sum_{i=0}^\infty e^{\beta f(ix)}h_\beta(ix)\right)\\
&\geq \lim_{\beta\to\infty} \frac{1}{\beta} \log\left(e^{\beta f(i_0x)}h_\beta(i_0x)\right)\\
&=  f(i_0x)+V(i_0x)=\sup_{j\in \mathbb{N}} (f(jx)+V(jx)).
\end{align*}

Consequently $V$ is a calibrated sub-action.
\end{proof}

Now we study the probability $\mu_\beta$ defined from $d\mu_\beta = h_\beta d\nu_\beta$. For $\beta \geq 1$ it is the $\sigma-$invariant Gibbs measure of $\beta f$. We prove below the existence of an exponential control of the constants in equation (\ref{Gibbseq})  with the parameter $\beta$. This will be very important in the proof of a Large Deviation Principle.

\begin{lemma}\label{gibbsinequality}
	There exist constants $\eta_1$ and $\eta_2$ such that for any $\beta > 1$, $x = (x_0,x_1,x_2,...)\in \Sigma_A$ and $m\geq 1$
	\begin{equation*}
	e^{\beta \eta_1}\leq \frac{\mu_\beta[x_0\cdots x_{m-1}]}{e^{\beta S_m(f(x)) - mP(\beta f)}} \leq e^{\beta \eta_2}.
	\end{equation*}
\end{lemma}

\begin{proof}
	We will follow ideas from \cite{Sarig03}. From Lemma  $\ref{propo2}$, there exist constants $C_1$ and $C_2$ such that
	$e^{\beta C_1}	\leq h_\beta \leq e^{\beta C_2}$. Let $M_2:=[\operatorname{Var}_1(f)+\sup_{n\geq 1}\operatorname{Var}_{n+1}S_nf]$ and let $\alpha$ be the minimal value of $f$ over the compact set $F^{\mathbb{N}}$. Then 
	\begin{align*}
	\mu_\beta[x_0\ldots x_{m-1}]&=\int \lambda_\beta^{-(m+K_0)}(L^{m+K_0}_{\beta f} h_\beta\cdot1_{[x_0\ldots x_{m-1}]})(z)d\nu_\beta(z)\\
	&\geq e^{C_1 \beta}\int \lambda_\beta^{-(m+K_0)} (L^{m+K_0}_{\beta f} 1_{[x_0\ldots  x_{m-1}]})(z)d\nu_\beta(z)\\
	&= e^{C_1 \beta}\int \lambda_\beta^{-(m+K_0)} \sum_{w_0...w_{K_0-1}}e^{S_{m+K_0}(\beta f)(x_0...x_{m-1}w_0...w_{K_0-1}z)} d\nu_\beta(z).
	\end{align*}
	(the summation above  is over all possible words $w_0...w_{K_0-1}$ such that $x_{m-1}w_0...w_{K_0-1}z$ is admissible)
		\begin{align*}
	\geq
	e^{C_1\beta}\lambda_\beta^{-(m+K_0)}&\int  \sum_{w_0...w_{K_0-1}}e^{S_{m}(\beta f)(x_0...x_{m-1}w_0...w_{K_0-1}z)} e^{S_{K_0}(\beta f)(w_0...w_{K_0-1}z)} d\nu_\beta(z)\\
	&\geq e^{\beta C_1 } \lambda_\beta^{-(m+K_0)}e^{S_m(\beta f)(x)+ K_0\beta \alpha-2\beta M_2 }
	\end{align*}
	
	In this way
	\[\frac{\mu_\beta[x_0\ldots x_{m-1}]}{e^{S_{m}(\beta f)(x)-mP(\beta f)}}\geq e^{\beta (C_1-2M_2  +K_0\alpha-K_0P(f))},\]
	where we use that $P(\beta f)\leq \beta P(f)$ if $\beta>1$.
	
	Besides, similarly we obtain 
	\begin{align*}
	\mu_\beta[x_0\ldots x_{m-1}] &\leq  e^{C_2 \beta}\int \lambda_\beta^{-(m+K_0)} \sum_{w_0...w_{K_0-1}}e^{S_{m+K_0}(\beta f)(x_0...x_{m-1}w_0...w_{K_0-1}z)} d\nu_\beta(z)\\
	&\leq e^{C_2 \beta}e^{S_m(\beta f)(x) +\beta M_2 }\lambda_\beta^{-m} \int \lambda_\beta^{-K_0} \sum_{i_0...i_{K_0-1}}e^{S_{K_0}(\beta f)(i_0...i_{K_0-1}z)} d\nu_\beta(z).
	\end{align*}
	
	(the last summation is over all possible words $i_1...i_{K_0-1}$ such that $ i_0...i_{K_0-1}z$ is admissible. This summation includes in general more terms than the summation on $w_0...w_{K_0-1}$)
	\begin{align*}
	&\leq e^{C_2 \beta} e^{S_m(\beta f)(x) +\beta M_2 }\lambda_\beta^{-m} e^{-C_1 \beta}\int \lambda_\beta^{-K_0} (L^{K_0}_{\beta f}h_\beta)(z) d\nu_\beta(z)\\
	&=e^{C_2 \beta}e^{S_m(\beta f)(x) +\beta M_2 }\lambda_\beta^{-m}e^{-C_1 \beta} \cdot 1.
	\end{align*}
	
	In this way
	\[\frac{\mu_\beta[x_0\ldots x_{m-1}]}{e^{S_m(\beta f)(x)-mP(\beta f)}} \leq e^{\beta (C_2-C_1+ M_2)}.\]
\end{proof}

\begin{corollary} If $\beta>1$ then the Gibbs measure $\mu_\beta$ is the equilibrium measure of $\beta f$.
\end{corollary}
\begin{proof}
It follows from the same arguments of \cite{MU} and \cite{Morris}  using our hypothesis and the above lemma.
\end{proof}

\begin{lemma}  \label{pressure}
	-
	
	\vspace{0.2cm}
	
	\noindent
	1. \, $\displaystyle{\lim_{\beta \to\infty} \frac{P(\beta f)}{\beta} = m(f);}$

	\vspace{0.2cm}

    \noindent
	2. \, $\displaystyle{\lim_{\beta \to\infty} \int f\, d\mu_{\beta} = m(f);}$
	
	\vspace{0.2cm}

	\noindent
	3. \,The family $(\mu_\beta)_{\beta>1}$ has an accumulation point when $\beta \to\infty$ (weak* topology). Furthermore, if $\mu_{\beta_j}\to \mu_{\infty}$ then $\mu_\infty$ is maximizing to $f$.

\end{lemma}

\vspace{0.3cm}

\begin{proof} 
We observe that for $0<\epsilon< \beta -1$,  
\begin{align*}
P((\beta-\epsilon)f) -(\beta -\epsilon)m(f) &\geq \int \beta f\,d\mu_{\beta f} - \int \epsilon f\, d\mu_{\beta f} +h(\mu_{\beta f}) -  (\beta -\epsilon)m(f)\\
&= P(\beta f) -\beta m(f) - \int \epsilon f\, d\mu_{\beta f} +\epsilon m(f)\\ 
& \geq P(\beta f) -\beta m(f).
\end{align*}
Then $\beta \mapsto [P(\beta f) -\beta m(f)], \beta >1$ is not increasing. As $P(\beta f) -\beta m(f) \geq 0$ we conclude that it is bounded.
Therefore
\[\lim_{\beta \to\infty} \frac{P(\beta f)}{\beta} = \lim_{\beta \to\infty} \frac{1}{\beta}(P(\beta f)-\beta m(f))+ m(f) = m(f).\]
This proves 1. 
\newline
 Fix $\beta_1<\beta_2$. Then
\[  \int \beta_2 f\, d\mu_{\beta_2} +h(\mu_{\beta_2}) >   \int \beta_2 f\, d\mu_{\beta_1} +h(\mu_{\beta_1})\]
and
\[\int \beta_1 f\, d\mu_{\beta_2} +h(\mu_{\beta_2}) <   \int \beta_1 f\, d\mu_{\beta_1} +h(\mu_{\beta_1}).  \]
Therefore 
\[\int (\beta_2 -\beta_1 ) f\, d\mu_{\beta_2} > \int (\beta_2 -\beta_1 ) f\, d\mu_{\beta_1}  \]
or equivalently 
\[ \int f\, d\mu_{\beta_2} > \int  f\, d\mu_{\beta_1}.     \]
This shows that $\mu_\beta(f)$ is increasing with $\beta$. 
Particularly, if $\beta_1<\beta_2$,	
	\begin{align*}
	h(\mu_{\beta_{1}}) &= P(\beta_{1}f) - \beta_{1}\mu_{\beta_{1}}(f)\\
	&=P(\beta_{1}f) -P(\beta_2 f) + P(\beta_2 f) - \beta_{1}\mu_{\beta_{1}}(f)+\beta_{1}\mu_{\beta_2}(f)\\ &\hspace{1cm}  -\beta_{1}\mu_{\beta_2}(f) +\beta_2\mu_{\beta_2}(f)-\beta_2\mu_{\beta_2}(f)\\
	&>(\beta_{1}-\beta_{2})\mu_{\beta_2}(f) + P(\beta_2 f) - \beta_{1}[\mu_{\beta_{1}}(f) -\mu_{\beta_2}(f)]+\\
	& \hspace{1cm} + (\beta_2-\beta_{1})\mu_{\beta_2}(f)-\beta_2\mu_{\beta_2}(f)\\
	&= P(\beta_2 f) - \beta_{1}[\mu_{\beta_{1}}(f) -\mu_{\beta_2}(f)] -\beta_2\mu_{\beta_2}(f)\\
	&=h(\mu_{\beta_2}) - \beta_{1}[\mu_{\beta_{1}}(f) -\mu_{\beta_2}(f)] > h(\mu_{\beta_2}).
	\end{align*}
	It follows that ${h(\mu_{\beta})}$ is decreasing and bounded below by zero, therefore using
	item 1, we get  
	\[ m(f) = \lim_{\beta\to\infty} \frac{P(\beta f)}{\beta} =\lim_{\beta\to\infty} \int f\, d\mu_{\beta } + \frac{h(\mu_{\beta})}{\beta} = \lim_{\beta\to\infty} \int f\, d\mu_{\beta }. \]
This proves 2.	
\newline	
The proof of existence of the limit for a subsequence of $(\mu_\beta)_{\beta>1}$ follows the same arguments found in \cite{JMU0}. As $f$ is not bounded we can not say directly that  $\mu_{\beta_j}(f) \to \mu_{\infty}(f)$ in order to prove that $\mu_{\infty}(f)=m(f)$. For each $n$ define $f_n(x)=\max\{f(x),-n\}$. Given $\epsilon>0$ there exists $j_0$ such that $\mu_{\beta_{j_0}}(f)> m(f)-\epsilon$. It follows that for $j\geq j_0$ and for any $n$, $\mu_{\beta_{j}}(f_n)> m(f)-\epsilon$.  Then  $\mu_{\infty}(f_n) \geq m(f) - \epsilon$ and applying the Monotone Convergence Theorem we get $\mu_\infty(f) \geq m(f)-\epsilon$.

\end{proof}

\begin{corollary}
	There exists a maximizing measure to $f$.
\end{corollary}

Now we prove some additional results concerning sub-actions and maximizing measures.  
In order to study some properties of the sub-actions is usual consider the Ma\~n\'e potential
	$$S_f(x,y):=\lim_{\epsilon \rightarrow 0}\sup_{n\geq 1 }\sup_{\substack{z:d(x,z)<\epsilon\\ \sigma^n(z)=y}}\{S_n(f-m(f))(z)\}$$
and the	Aubry set
	$$\Omega(f,\sigma):=\{x\in\Sigma_A(\mathbb{N})\, |\, S_f(x,x) = 0\}. $$}

\begin{proposition}\label{propo-propiedades}
	Properties of the Ma\~n\'e potential: 
	\begin{enumerate}
		\item $S_f(x,y)\leq V(y)-V(x)$, for any bounded calibrated sub-action $V$ and $x,y\in \Sigma_A(\mathbb{N})$,
		\item $S_f(x,x)\leq 0$, for every $x\in \Sigma_A(\mathbb{N})$,
		\item $S_f(x,y)+S_f(y,z)\leq S_f(x,z)$ for any $x,y,z \in \Sigma_A(\mathbb{N})$,
		\item  $\operatorname{supp}(\mu) \subset \Omega(f,\sigma)$ for any $\mu\in \mathcal{M}_{\max}(f)$.
		\end{enumerate}
\end{proposition}

\begin{proof}
From the definition of calibrated sub-action we have
\[S_n(f-m(f))(z)\leq V(\sigma^n(z)) - V(z)\]
for any $z$ and consequently 1. and 2.\newline
In order to prove 3. we fix an error $\rho>0$. Let $k$ be sufficiently large such that $\sup_n Var_{n+k}S_n(f) <\rho$ (from the Walters property) and consider points $y_z=y_z(k)$, $x_y=x_y(k)$ and numbers $n_1=n_1(k)$, $n_2=n_2(k)$ such that
\[d(x_y,x)\leq r^{k},\,\, \sigma^{n_1}(x_y)=y,\,\,S_{n_1}(f-m(f))(x_y)\geq S_{f}(x,y)-\rho \]
and
\[d(y_z,y)\leq r^{k},\,\, \sigma^{n_2}(y_z)=z,\,\,S_{n_2}(f-m(f))(y_z)\geq S_{f}(y,z)-\rho.  \]  	
Writing $x_y=(a_1a_2...a_{n_1}y)$ set $x_z:=(a_1a_2...a_{n_1}y_z)$.
Then $d(x_z,x)\leq r^{k},\,\,\sigma^{n_1+n_2}(x_z)=z$ and
\begin{align*}
S_{n_1+n_2}(f-m(f))(x_z) &= S_{n_1}(f-m(f))(x_z)+ S_{n_2}(f-m(f))(y_z) \\
&\geq (S_{n_1}(f-m(f))(x_y) -\rho) + S_{n_2}(f-m(f))(y_z)\\
&\geq S_{f}(x,y)+S_{f}(y,z) -3\rho.
\end{align*}
Making $k\to +\infty$ we obtain 
\[S_f(x,z) \geq S_{f}(x,y)+S_{f}(y,z) -3\rho.\]
This concludes the proof of 3. \newline
The proof of 4. is a consequence of Lemma 2.2 in \cite{Mane}. If $p \in \operatorname{supp}(\mu)$ then $p$ belongs to the support of some ergodic maximizing measure, thus, we can suppose $\mu$ ergodic. Given $\epsilon =r^{k}$ for some $k>0$, as $p\in\operatorname{supp}(\mu)$,  $\mu(B(p,\epsilon))>0$. There exist $z \in B(p,\epsilon)$ and $n\in \mathbb{N}$ such that $\sigma^n(z) \in B(p,\epsilon)$ and $|S_{n}f(z)-n\mu(f)| \leq \epsilon$. Thus, writing $z=(z_1z_2z_3...)$ and defining $z'=(z_1...z_np)$ we have $d(z,z')\leq r^{n+k}$. Then $d(z',p)\leq \epsilon$, $\sigma^{n}(z')=p$ and
\[S_{n}f(z')-n\mu(f) \geq S_{n}f(z)-n\mu(f) - \sup_l \operatorname{Var}_{k+l}(S_l f) \geq -r^k - \sup_l 
\operatorname{Var}_{k+l}(S_l f) \]
When $k \to +\infty$ we obtain $S_f(p,p) \geq 0 $ 
and applying the item 2. we conclude the proof.     

\end{proof}

\begin{proposition}\label{twosubaction}
 If the maximizing measure $\mu$ to $f$ is unique, then 
for any fixed $a\in supp(\mu)$ and bounded continuous calibrated sub-action $V$,
\[V(x) = V(a) + S_f(a,x) \,\,\,\,\,\forall \,\, x\in \Sigma_A(\mathbb{N}).\]
Particularly, if the maximizing measure is unique, any two bounded calibrated sub-actions differ by a constant.
\end{proposition}
\begin{proof}
	We follow the ideas of \cite{BLT}. Let $V$ be any bounded calibrated sub-action and define $R=f +V -V\circ \sigma - m(f)$. By definition of calibrated sub-action we have $R \leq 0$.  Given $x_0 \in \Sigma_A(\mathbb{N})$, consider a sequence of points $x_n$ such that $\sigma(x_n) = x_{n-1}$ and $R(x_n) = 0,\,n=1,2,3,...$ (this sequence exists because $V$ is calibrated) and define a sequence of probabilities $\nu_n$ by
	$\nu_n = \frac{\delta_{x_1} +\cdots +\delta_{x_n}}{n}$. As $V$ is bounded and $f$ is coercive, there exists a finite set $B\subset \mathbb{N}$ such that any $x_n$ is of the form $b_n\ldots b_1x_0,\,b_i\in B,\, n\geq 1$. This shows that $\nu_n$ has an accumulation measure $\nu$ (in the weak* topology) when $n\to\infty$.
	It is easy to see that $\nu$ is $\sigma$-invariant. Also if $R_k(x):=\max\{R(x),-k\}$ then $R_k$ is continuous, bounded and $\int R_k\, d\nu = 0$, because $\int R_k\, d\nu_n = 0$ for all $n$. Applying the Monotone Convergence Theorem we conclude that $\int R\, d\nu = 0$. Then $\nu$ maximizes the integral of $R$ and therefore the integral of $f$. 	
	Suppose $\mu$ is the unique maximizing measure to $f$. Particularly $\nu=\mu$ and we conclude that $\mu$ is an accumulation measure of $\nu_n$. Fix a point $a\in \operatorname{supp}(\mu)$.
	It follows that, for large enough $n$, $\nu_n(\{z:d(z,a)\leq \epsilon\}) >0$. Then $a$ is an accumulation point of the sequence $\{x_n\}$. Writing $x_{n_k}\to a$, particularly we get
	\begin{align*}
	V(x_0)-V(a)\geq S_f(a,x_0) &\geq \lim_{n_k\to\infty}[S_{n_k}(f)(x_{n_k}) -n_km(f)]\\
	&=\lim_{n_k\to\infty} [S_{n_k}(R)(x_{n_k}) -V(x_{n_k})+V(x_0)]\\
	& =V(x_0)-V(a).
	\end{align*}
	
\end{proof}

\section{Large Deviation Principle}
In this section we prove a Large Deviation Principle for the 
family $(\mu_\beta)_{\beta}$. 

From now on  we suppose that $\Sigma_A(\mathbb{N})$ is a sub-shift finitely primitive and $f:\Sigma_A(\mathbb{N})\to\mathbb{R}$  satisfies the Walters condition, with $\operatorname{Var}_1(f)<\infty$ and $P(f)<\infty$.
Furthermore,  we assume that  $f$ has a unique maximizing measure $\mu$.






From Lemma $\ref{twosubaction}$ we know that any two {bounded} calibrated sub-actions differ by a constant.
Define
\begin{align*}
g_\beta&:=\beta f+\log h_\beta-\log (h_\beta \circ \sigma)-P(\beta f)\\
R_ {-}&:= f + V - V\circ\sigma -m(f)\\
R=R_{+} &:= -f -V + V\circ\sigma +m(f),\\
\end{align*}
where $V$ is any bounded calibrated sub-action to $f$ and $m(f)=\int f\,d\mu$. 

Observe that $R=R_{+}\geq 0$ because $V$ is a calibrated sub-action. 
 We write
 $R^n(x):=\sum\limits_{j=0}^{n-1}R(\sigma^j(x))$
and $R^\infty:=\lim\limits_{n \rightarrow \infty}R^n$. 
The function $R^{\infty}$ can {attain} the value $+\infty$ in some points.

\begin{lemma}\label{convergeR}
The functions $\frac{g_{\beta}}{\beta}$ converge to $R_{-}$ uniformly  in compact sets.
There exists a sequence of numbers $\{M_n\}$ such that for $n$ large enough
\[\operatorname{Var}_n(g_\beta /\beta) \leq M_n,\,\,\,\,\operatorname{Var}_n(R_{-}) \leq M_n \,\,\,\,\text{and}\,\,\,\, \lim_{n\to\infty}M_n=0.\] 
Furthermore
\[\operatorname{Var}_{n+k}(R_{-}^n) \leq \operatorname{Var}_{n+k}(S_n(f)) +M_{n+k} +M_k.\]

\end{lemma}

\begin{proof}
For any fixed $x$, the family $\frac{g_{\beta}(x)}{\beta}$ is bounded in the variable $\beta$ and from Proposition \ref{propo-subaction} any accumulation point has the form $R_{-}(x)=f(x) + V(x) - V(\sigma(x)) - m(f)$ where $V$ is any bounded calibrated sub-action. 

It follows from Corollary \ref{equicontinuous} that for a fixed compact $K \subset \Sigma_A(\mathbb{N})$ we have $g_\beta/\beta$ bounded, equicontinuous and converging pointwise to $R_{-}$. Therefore converges uniformly on $K$.

{
From Corollary $\ref{variacaoh}$,  there exists a sequence of numbers $\{M_n\}_{n\geq 3}$ such that
\[\operatorname{Var}_n(f)+\operatorname{Var}_n(\frac{1}{\beta}\log(h_\beta))+\operatorname{Var}_n(\frac{1}{\beta}\log(h_\beta\circ\sigma)) \leq M_n \,\,\,\,\text{and}\,\,\,\, \lim_{n\to\infty}M_n=0.\] 
As $g_\beta /\beta = f + \frac{1}{\beta}\log(h_\beta)-\frac{1}{\beta}\log(h_\beta\circ\sigma)-\frac{1}{\beta}\log(P(\beta))$, we obtain 
\[\operatorname{Var}_n(g_\beta/\beta) \leq M_n \,\,\,\, and \,\,\,\,\operatorname{Var}_n(R_{-}) \leq M_n.\] }
Also, as
\[R_-^n = S_n(f)  +V - V\circ\sigma^{n} - n\cdot m(f),\]
from Corollary \ref{variacaoh} again, we conclude that
\[\operatorname{Var}_{n+k}(R_-^n) \leq \operatorname{Var}_{n+k}(S_n(f)) +M_{n+k} +M_k.\]

\end{proof}

From Lemma \ref{gibbsinequality}, for any cylinder $C$, the family $\frac{1}{\beta}\log \mu_\beta(C)$ is bounded. There is a countable set of cylinders, therefore, applying a Cantor's diagonal argument,  there exists a 
subsequence $\beta_i$ such that, for any cylinder $C$,  $\frac{1}{\beta_i}\log \mu_{\beta_i}(C)$ converges.

\begin{proposition}\label{devfunction}
Under the same assumptions of Theorem $\ref{teoremaA}$ and besides suppose that for any
cylinder $C$ there exists  $\lim\limits_{\beta_i\rightarrow \infty} \frac{1}{\beta_i}\log\mu_{\beta_i}$
then
\begin{enumerate}
  \item The function $$\tilde{I}(x):=-\lim_{n\rightarrow \infty} \lim_{\beta_i\rightarrow \infty} \frac{1}{\beta_i}\log\mu_{\beta_i}([x_0\ldots x_{n}])\geq 0 $$
is well defined and lower semi-continuous.
  \item $$ \lim_{\beta_i\rightarrow \infty} \frac{1}{\beta_i}\log\mu_{\beta_i}(C)=-\inf_{x\in C}\tilde{I}(x) \ \
\text{for all cylinder} \ C. $$
  \item $\tilde{I}(x)\geq \tilde{I}(\sigma(x))$. Particularly, exists $\tilde{I}_0(x):=\lim\limits_{n\rightarrow \infty} \tilde{I}(\sigma^n(x))$.
\end{enumerate}
\end{proposition}
\begin{proof}
We will prove the item 2. The proofs of items 1. and 3. are consequence of the definition of $\tilde{I}$ and can be found in \cite{Mengue2}. 

\bigskip

{Given $x=(x_0,x_1,x_2...)\in C$,} 
\begin{align*}
\lim_{\beta_i\rightarrow \infty} \frac{1}{\beta_i}\log\mu_{\beta_i}(C)&=
\lim_{n\rightarrow \infty}\lim_{\beta_i\rightarrow \infty} \frac{1}{\beta_i}\log\mu_{\beta_i}(C)\\
&\geq \lim_{n\rightarrow \infty}\lim_{\beta_i\rightarrow \infty} \frac{1}{\beta_i}\log\mu_{\beta_i}([x_0\ldots x_n])\\
&= -\tilde{I}(x),
\end{align*}
Taking any possible $x$ we obtain $$\lim_{\beta_i\rightarrow \infty} \frac{1}{\beta_i}\log\mu_{\beta_i}(C) \geq -\inf_{x\in C} \tilde{I}(x).$$

Now we prove the other inequality.
Indeed, put $C=[x_0\ldots x_m]$ for some $m$.
From Lemma \ref{gibbsinequality} we can take constants {$a$ and $b$, depending on $f, x_0,...,x_m$ such that}

\begin{equation*} \label{cap3-ec2}
e^{\beta a}e^{\beta \inf f|_{[i]}}\leq \mu_{\beta}[x_0\ldots x_m i]\leq e^{\beta b}e^{\beta \sup f|_{[i]}}.
\end{equation*}
As $\operatorname{Var}_1(f)<\infty$ we have $\inf f|_{[0]}> -\infty$. Then from Lemma \ref{lemmasumabble} there exists some $l_n$ such that
\[\sum_{i\geq l_n}e^{ b}e^{ \sup f|_{[i]}} \leq e^{a}e^{\inf f|_{[0]}}.\]
Particularly, for any $\beta>1$
\[ \sum_{i\geq l_n}e^{\beta b}e^{\beta \sup f|_{[i]}}\leq \left(\sum_{i\geq l_n}e^{ b}e^{ \sup f|_{[i]}}\right)^{\beta} \leq e^{\beta a}e^{\beta \inf f|_{[0]}}.\]
Therefore, for any $\beta>1$
\[\sum_{i\geq l_n}\mu_{\beta}[x_0\ldots x_m i] \leq \mu_{\beta}[x_0\ldots x_m 0]\leq \sum_{i=0}^{l_n-1}\mu_{\beta}[x_0\ldots x_m i].\]
Then

\begin{align*}
\lim_{\beta \to\infty}\frac{1}{\beta}\log \mu_{\beta}[x_0\ldots x_m]& = \max\left\{\lim_{\beta \to\infty}\frac{1}{\beta}\log \sum_{i\geq l_n}\mu_{\beta}[x_0\ldots x_m i], \lim_{\beta \to\infty}\frac{1}{\beta}\log \sum_{i=0}^{l_n-1}\mu_{\beta}[x_0\ldots x_m i]\right\}\\
&=\lim_{\beta \to\infty}\frac{1}{\beta}\log \sum_{i=0}^{l_n-1}\mu_{\beta}[x_0\ldots x_m i]\\
&=\max_{0\leq i< l_n}\left\{\lim_{\beta \to\infty}\frac{1}{\beta}\log \mu_{\beta}[x_0\ldots x_m i]\right\}.
\end{align*}

Particularly, there exists some $y_0\in \mathbb{N}$ such that
\[\lim_{\beta \to\infty}\frac{1}{\beta}\log \mu_{\beta}[x_0\ldots x_m] =\lim_{\beta \to\infty}\frac{1}{\beta}\log \mu_{\beta}[x_0\ldots x_my_0].\]

Using an inductive argument  we can construct a sequence   $y_0,y_1,\ldots$ of elements in $\mathbb{N}$
such that
\begin{align*}
\lim_{\beta_i \rightarrow \infty}\frac{1}{\beta_i} \log \mu_{\beta_i}(C)&=
 \lim_{\beta_i \rightarrow \infty}\frac{1}{\beta_i} \log \mu_{\beta_i}([x_0\ldots x_m y_0])\\
&= \lim_{\beta_i \rightarrow \infty}\frac{1}{\beta_i} \log \mu_{\beta_i}([x_0\ldots x_m y_0y_1])= \cdots
\end{align*}

Let $z$ be the intersection point of the cylinder sets $[x_0\ldots x_m y_0\ldots y_n]$, $n\geq 1$.
Thus, when $n \rightarrow \infty$
\begin{align*}
\lim_{\beta_i \rightarrow \infty}\frac{1}{\beta_i} \log \mu_{\beta_i}(C)&=
\lim_{n \rightarrow \infty}\lim_{\beta_i \rightarrow \infty}\frac{1}{\beta_i} \log \mu_{\beta_i}([x_0\ldots x_m y_0\ldots y_n])\\
&=-\tilde{I}(z)\\
&\leq -\inf_{x \in C}\tilde{I}(x).
\end{align*}

\end{proof}

\begin{lemma}\label{lemaA} For any $\beta >1$,
  
$$\mu_{\beta}([x_1\ldots x_n])=\int_{[x_0\ldots x_n]}e^{-g_{\beta}}d\mu_\beta.$$

Particularly we obtain
\begin{equation*} \label{ec44}
\bigg(\inf _{z\in [x_0\ldots x_n]} e^{-g_\beta(z)}\bigg)\mu_\beta([x_0\ldots x_n])\leq \mu_\beta([x_1\ldots x_n])\leq
\bigg(\sup _{z\in [x_0\ldots x_n]} e^{-g_\beta(z)}\bigg)\mu_\beta([x_0\ldots x_n]).
\end{equation*}

\end{lemma}
The proof follows the same computations present in the proof of Proposition 3.2 in \cite{PP}.

\begin{proposition}\label{propoB}
\begin{equation*}
\lim_{n \rightarrow \infty} \lim_{\beta \rightarrow \infty} \frac{1}{\beta}
\log \frac{\mu_\beta[x_1\ldots x_n]}{\mu_\beta[x_0\ldots x_n]}=R_+(x),
\end{equation*}
where $x=(x_0,x_1,x_2,...)$

\end{proposition}

\begin{proof}
{Given $x=(x_0,x_1,x_2,...)$,} applying  Lemma $\ref{lemaA}$ to the function $g_\beta$, we have
\begin{equation*} \label{ec38}
\inf_{z\in [x_0\ldots x_n]}-\frac{g_\beta(z)}{\beta}\leq
\frac{1}{\beta}\log \frac{\mu_\beta[x_1\ldots x_n]}{\mu_\beta[x_0\ldots x_n]}\leq \sup_{z\in [x_0\ldots x_n]}-\frac{g_\beta(z)}{\beta}.
\end{equation*}

From Lemma \ref{convergeR}, there exist constants $M_n$ converging to zero when $n\to\infty$ such that
\begin{equation*}
\sup_{z\in [x_0\ldots x_n]}-\frac{g_\beta(z)}{\beta} \leq -\frac{g_\beta(x)}{\beta}+M_n
\,\,\,\,\,\text{and}\,\,\,\,\,
\inf_{z\in [x_0\ldots x_n]}-\frac{g_\beta(z)}{\beta} \geq -\frac{g_\beta(x)}{\beta}-M_n.
\end{equation*}

Thus 
\begin{equation*}
-M_n-\frac{g_\beta(x)}{\beta}\leq \frac{1}{\beta}\log \frac{\mu_\beta[x_1\ldots x_n]}{\mu_\beta[x_0\ldots x_n]}
\leq M_n-\frac{g_\beta(x)}{\beta}.
\end{equation*}
Taking the limit on $\beta$ and then the limit on $n$ we get

\begin{equation*}
\lim_{n \rightarrow \infty} \lim_{\beta \rightarrow \infty} \frac{1}{\beta}
\log \frac{\mu_\beta[x_1\ldots x_n]}{\mu_\beta[x_0\ldots x_n]}=R_+(x).
\end{equation*}

\end{proof}

In the next proofs we make strong use of the uniqueness hypothesis of the maximizing measure.

\begin{proposition}\label{propoA}
Let $x \in \Sigma_A$ such that $R_+^\infty(x)<\infty$ and $p \in \operatorname{supp}(\mu)$. Then $p$ is an accumulation point of $\sigma^n(x)$.
\end{proposition}

\begin{proof}
We define for each  $n\in \mathbb{N}$ the
probability measure
$\mu_n=\frac{1}{n}\sum_{j=0}^{n-1}\delta_{\sigma^j(x)}.$
Using that $f$ is coercive and $V$ is bounded we get that $R_-= f+ V-V\circ\sigma -m(f)$ is coercive. Particularly, as $R_+^\infty(x)<\infty$ we conclude that $x=(x_0,x_1,x_2,...)$ where all the symbols $x_i$ belong to a finite alphabet $B$. This shows that the probabilities $\mu_n$ have support in a same compact set. Particularly, there exist accumulation measures in the weak* topology. It is easy to see that any accumulation measure is invariant. If $M=R_{+}^\infty(z)$ then $$0\leq \int R_{+}\,d\mu_n = \frac{1}{n}R_{+}^{n}(x) \leq \frac{M}{n}\to 0.$$ Therefore, the accumulation measure maximizes the integral of $f$ and so coincides with $\mu$.  Thus, for a fixed point $p \in \operatorname{supp}(\mu)$,  $\mu_n(\{x\in \Sigma_A(\mathbb{N}):d(x,p)\leq \epsilon)$ is positive for sufficiently large $n$. Therefore $p$ is an accumulation point of $\sigma^n(x)$.

\end{proof}
In order to conclude the proof of the Theorem \ref{teoremaA} we present the following proposition.

\begin{proposition}
For any cylinder $C$ we have
\begin{equation*}
\lim_{\beta \rightarrow \infty}\frac{1}{\beta}\log \mu_{\beta}(C)=-\inf_{x\in C}\sum_{j=0}^\infty R_+(\sigma^j(x)).
\end{equation*}
\end{proposition}
\begin{proof}
{Fix a cylinder $C$. From Lemma \ref{gibbsinequality} the family  $\frac{1}{\beta} \log \mu_{\beta}(C)$ is bounded. We will show that any accumulation point  is equal to $-\inf\limits_{x\in C}\sum\limits_{j=0}^\infty R(\sigma^j(x))$. In this way we suppose that $\frac{1}{\beta_i} \log \mu_{\beta_i}(C)$ converges. From Cantor's diagonal argument we can suppose the limit exists for any cylinder set. Then we assume that the probabilities $\mu_{\beta_i}$ satisfies a $L.D.P.$ with deviation function $\tilde{I}$ given in Proposition \ref{devfunction}.  In order to finish the proof we need to show that}
	\begin{equation*}
	\tilde{I}(x)=R_+^\infty(x).
	\end{equation*}

	
 Proposition $\ref{propoB}$ implies
\begin{equation*}
\lim_{n \rightarrow \infty} \lim_{\beta_i \rightarrow \infty} \frac{1}{\beta_i}
\log \frac{\mu_{\beta_i}[x_1\ldots x_n]}{\mu_{\beta_i}[x_0\ldots x_n]}=R_+(x).
\end{equation*}
Then
\begin{equation*}
R_+=\lim_{n \rightarrow \infty}
\big( \lim_{\beta_i \rightarrow \infty}\frac{1}{\beta_i}
\log \mu_{\beta_i}[x_1\ldots x_n]-\lim_{\beta_i \rightarrow \infty} \frac{1}{\beta_i}
\log \mu_{\beta_i}[x_0\ldots x_n]\big)
\end{equation*}
and so
\begin{equation*}
\tilde{I}(x)=R_+(x)+\tilde{I}(\sigma(x)).
\end{equation*}

Therefore, for each  $n$
\begin{equation*}
\tilde{I}(x)=R_+^n(x)+\tilde{I}(\sigma^n(x)),
\end{equation*}
so, from Proposition \ref{devfunction}, if  $n \rightarrow \infty$ we get
\begin{equation*}
\tilde{I}(x)=R_+^\infty(x)+\tilde{I}_0(x).
\end{equation*}
This concludes the proof that  $\tilde{I}(x)\geq R_+^\infty(x)$.

In order to proof the opposite inequality we can suppose
$R_+^\infty(x)<\infty$.
Fix $x=(x_0,x_1,x_2,\ldots)$ such that $R_+^\infty(x)<\infty$ and
 $y \in \operatorname{supp}(\mu)$. Observe that $\tilde{I}(y)=0$.
Furthermore, from Proposition \ref{propoA},  $y$
 is an accumulation point of the sequence $\{\sigma^n x\}$. Then, fixed $k$, there exists  $m$
 such that  $x_{m}x_{m+1}\ldots x_{m+k-1}=y_0y_1\ldots y_{k-1}$.
As $\tilde{I}$ is lower semi-continuous,
\begin{equation}\label{desigualdadeIR}
\tilde{I}(x)\leq \liminf_{m \rightarrow \infty}\tilde{I}(x_0\ldots x_{m-1}y) =  \liminf_{m \rightarrow \infty}(R_+^m(x_0\ldots x_{m-1}y)).
\end{equation}

Applying the Lemma \ref{convergeR} we get
\begin{align*}
R_+^m(x_0\ldots x_{m-1}y)&\leq R_+^m(x)+\sup_{n\geq 1}\operatorname{Var}_{n+k}(R_+^{n})\\
&\leq R_+^{\infty}(x)+\sup_{n\geq 1}[\operatorname{Var}_{n+k}(S_n(f)) +M_{n+k} +M_k].
\end{align*}

Thus, letting $k \rightarrow \infty$  and using (\ref{desigualdadeIR})  we obtain
\begin{equation*}
\tilde{I}(x)\leq R^\infty(x).
\end{equation*}

\end{proof}

We point out that in \cite{CLO} is presented an example due to R. Leplaideur where the maximizing measure $\mu$ of potential $f$ is unique, but there is a point $x$ outside of  
the $\operatorname{supp}{(\mu)}$ satisfying $R_+^{\infty}(x) = 0$. It follows that there are examples where $I$ can be zero outside of the support.

\begin{proposition}
Suppose that $f$ has a unique maximizing measure $\mu$. Then
\[ \Omega(f,\sigma)= \{x \in \Sigma_A(\mathbb{N}): I(x)=0\}.  \]	
\end{proposition}

\begin{proof} We can suppose $m(f)=0$.
Let $x \in 	\Omega(f,\sigma)$, that means, $S_f(x,x)=0$. Then there exist $x_k \to x$ and $n_k \to +\infty$ such that
\[ \sigma^{n_k}(x_k)=x,\,\, \lim_{k \to +\infty} \sum_{j=0}^{n_k-1}f(\sigma^{j}(x_k))=0. \]
Let $V$ be any bounded calibrated sub-action. For each fixed $m$ we will show that $f(\sigma^m(x))+V(\sigma^m(x))-V(\sigma^{m+1}(x))=0$, proving that $I(x)=0$.
As $n_k\to +\infty$ we can suppose $m< n_k$. Then  
\begin{align*}
V(\sigma^m(x))  &= V\circ\sigma(\sigma^{n_k+m-1}(x_k))  \\
&\geq V(\sigma^{m+1}(x_k)) + \sum_{j=m+1}^{n_k+m-1}f(\sigma^{j}(x_k)) \\
&= V(\sigma^{m+1}(x_k)) +  \sum_{j=0}^{n_k-1}f(\sigma^{j}(x_k)) - \sum_{j=0}^{m}f(\sigma^{j}(x_k))+\sum_{j=n_k}^{n_k+m-1}f(\sigma^{j}(x_k)) \\
&= V(\sigma^{m+1}(x_k)) +  \sum_{j=0}^{n_k-1}f(\sigma^{j}(x_k)) - \sum_{j=0}^{m}f(\sigma^{j}(x_k)) +\sum_{j=0}^{m}f(\sigma^{j}(x)) - f(\sigma^{m}(x)). 
\end{align*}	
Making $k\to \infty$ we obtain
\[ V(\sigma^m(x))  \geq V(\sigma^{m+1}(x)) -  f(\sigma^m(x)). \]
As $V$ is a calibrated sub-action, the opposite inequality is also verified, therefore   $f(\sigma^m(x))+V(\sigma^m(x))-V(\sigma^{m+1}(x))=0$. This proves that $I(x)=0$.	

Now we suppose that $I(x)=0$. We want to show that $S_{f}(x,x)=0$.
From Proposition \ref{propoA} there exist $n_k \to \infty$ and $b \in \operatorname{supp}(\mu)$ such that $\sigma^{n_k}(x)\to b$. From Proposition \ref{propo-propiedades}, $b \in \Omega(f,\sigma)$. Let $V(z) :=S_f(b,z)$. It is a bounded calibrated sub-action by Proposition \ref{twosubaction}.

We write $x=(x_1,x_2,x_3,...)$ and define $b_{n_k}:=(x_1,...,x_{n_k},b)$. Then $\sigma^{n_k}(b_{n_k})=b$ and as $d(\sigma^{n_k}(x),b)\to 0$, we can suppose $d(b_{n_k},x)< r^{n_k+2}$. It follows that
\begin{align*}
S_f(x,b) &\geq \limsup_{k\to\infty} S_{n_k}f(b_{n_k}) \geq \limsup_{k\to\infty} (S_{n_k}f(x) - \sup_{l\geq 1}\operatorname{Var}_{n_k+l}(S_{n_k}f)) \\
&= \limsup_{k\to\infty} (S_{n_k}f(x)) \stackrel{I(x)=0}{=} \limsup_{k\to\infty}(V(\sigma^{n_k}(x)) -V(x)))=V(b)-V(x)  . 
\end{align*}
As by definition $V(\cdot)=S_f(b,\cdot)$, we obtain
\[S_f(x,b)\geq S_f(b,b)-S_f(b,x) .\]
Applying the Proposition \ref{propo-propiedades} again, 
\[ 0 =S_{f}(b,b) \leq S_{f}(x,b)+S_f(b,x)  \leq S_f(x,x) \leq 0.\]
This shows that $S_f(x,x)=0$ and completes the proof.	
\end{proof}

\section{ Beyond the BIP property and concluding remarks}

 We should ask what it is possible to do beyond the 
BIP case. 
We already know that for transitive Markov shifts over
the alphabet $\N$ we have at most one
equilibrium measure \cite{Daon, Sarig04} for the class of potentials
defined for us. But, eventually a measure could not exist if
the temperature is too low. More specifically we know that:

\begin{theorem} [Sarig, \cite{Sarig05}]
  Let $\Sigma_A({\N})$ be a Renewal shift and  $f$ locally H\"older 
  continuous such that
  $\sup f<\infty$. Then, there exists a constant $\beta_c \in (0,\infty]$ such
  that
  \begin{itemize}
    \item For $0<\beta<\beta_c$ there exists a probability measure 
     $\mu_\beta$
    corresponding to $\beta f$. For $\beta>\beta_c$ 
    there is no an equilibrium probability measure  
    corresponding to $\beta f$.
    \item $P(\beta f)$ is real analytic in $(0,\beta_c)$
    and linear in $(\beta_c,\infty)$. At $\beta_c$,
    $P(\beta f)$ is continuous but not analytic. 
  \end{itemize}
  \end{theorem}

  \begin{theorem}[Iommi, \cite{Iommi}]\label{Iommi}
  Let $\Sigma_A({\N})$ be a Renewal shift and $f$ locally H\"older continuous with
  $\sup f<\infty$. Then
  \begin{itemize}
    \item For $\beta_c=\infty$, there exists a measure $f$-maximizing.
    \item If $\beta_c<\infty$, there is no a measure $f$-maximizing.
  \end{itemize}
 \end{theorem}

\begin{theorem} [R. Bissacot and R. Freire Jr, \cite{BF}]\label{Bissacot}
Let $\sigma$ be the shift on $\Sigma_{A}(\N)$ with $A$ irreducible, $f: \Sigma_{A}(\N) \to \R$ be a function with bounded variation and coercive. Then, there is a finite set $\mathcal A \subset \N$ such that $A|_{\mathcal A \times \mathcal A}$ is irreducible and
\[
\sup_{\mu \in \mathcal M_{\sigma}(\Sigma_{A}(\N ))} \int f \ d\mu =\sup_{\mu \in \mathcal M_{\sigma}(\Sigma_{A}(\mathcal A ))} \int f \ d\mu .
\]
Furthermore, if $\nu$ is a maximizing measure, then
$\text{supp }\nu \subset  \mathcal M_{\sigma}(\Sigma_{A}(\mathcal A )).$
\end{theorem}
 
Notice that the Renewal shifts are contained in the class of 
potentials where the Theorem $\ref{Bissacot}$ holds, thus, if we 
combine the two results, Theorem  $\ref{Iommi}$ and  Theorem $\ref{Bissacot}$,  
we conclude that there is no phase transition for  H\"older
and coercive potentials in Renewal shifts, 
therefore, it is expected that in this case there exists a Large
Deviation Principle as well.


\section*{Acknowledgments}

The authors are very grateful to Renaud Leplaideur who pointed out a gap in the proof of Proposition
$\ref{propo-subaction}$.

{\footnotesize


\begin{thebibliography}{99}

\bibitem{BLT} A. Baraviera, A. O. Lopes and P. Thieullen,
 A Large Deviation Principle for the equilibrium states of H\"older Potentials: the zero temperature case,
 \emph{Stochastics and  Dynamics} \textbf{6}(2006), 77-96.

 \bibitem{BLL}
A. Baraviera, R. Leplaideur and  A.O. Lopes, Ergodic optimization, zero temperature limits and the Max-Plus algebra, 29 Col{\'o}quio Brasileiro de Matem{\'a}tica, IMPA, (2013).

 \bibitem{BLM} A. Baraviera, A.O. Lopes and J. K. Mengue, On the selection of subaction and measure for a
subclass of potentials defined by P. Walters, \emph{Ergodic Theory and Dynamical Systems} \textbf{33}(2013),
1338-1362.

\bibitem{BG}R. Bissacot and E. Garibaldi, Weak KAM methods and ergodic optimal problems for countable Markov shifts,
 \emph{Bulletin of the Brazilian Mathematical Society} \textbf{41} (2010), 321-338.

 \bibitem{BF}R. Bissacot and  R. Freire Jr, On the existence of maximizing measures for irreducible countable Markov shifts: a dynamical proof,
 \emph{Ergodic Theory and Dynamical Systems} \textbf{34} (2014), 1103-1115.

\bibitem{BGT} R. Bissacot, E. Garibaldi, Ph. Thieullen. Zero-temperature phase diagram for double-well type potentials in the summable variation class. To appear in \emph{Ergodic Theory and Dynamical Systems}, preprint (2016), 	arXiv:1512.08071.

\bibitem{Bou}  T. Bousch. La condition de Walters. \emph{Annales scientifiques de l'\'Ecole Normale Sup\`erieure} \textbf{34},  Issue: 2, (2001), 287-311.

\bibitem{Bo} R. Bowen, Equilibrium States and the Ergodic Theory of Anosov Diffeomorphisms. 2nd Edition. Edited by Jean-Ren\'e Chazottes. Springer-Verlag, Berlin, (2008).

\bibitem{Bremont}J. Br\'emont, Gibbs measures at temperature zero,
\emph{Nonlinearity} \textbf{16}(2)(2003),419-426.

\bibitem{Sarig04} J. Buzzi, O. Sarig, Uniqueness of equilibrium measures for
countable Markov shifts and multidimensional
piecewise expanding maps, \emph{Ergodic Theory and Dynamical Systems} \textbf{23}(2003), 1383-1400

\bibitem{Chazottes}J. R. Chazottes and M. Hochman, On the zero-temperature limit of Gibbs states,
\emph{Comm. Math. Phys} \textbf{297}(1)(2010),265-281.

\bibitem{CLO} G. Contreras, A. O. Lopes e E. R. Oliviera. Ergodic Transport Theory, Periodic Maximizing Probabilities and the Twist Condition. Modeling, Dynamics, Optimization and Bioeconomics I. Springer Proceedings in Mathematics $\&$ Statistics Vol. \textbf{73} (2014), 183-219.

\bibitem{CLT} G. Contreras, A. O. Lopes and Ph. Thieullen, Lyapunov minimizing measures for expanding maps of the circle,
\emph{Ergodic Theory and Dynamical Systems} \textbf{21} (2001), 1379-1409.

\bibitem{GC}G. Contreras. Ground states are generically a periodic
orbit. To appear in \emph{Inventiones mathematicae} (2015), doi: 10.1007/s00222-015-0638-0.

\bibitem{CG} J.-P. Conze and Y. Guivarc'h. Croissance des sommes ergodiques et principe variationnel, manuscript (1993).


\bibitem{Daon} Y. Daon, Bernoullicity of equilibrium measures on countable Markov shifts. Disc. Cont. Dyn. Sys. \textbf{33}, (2013), 4003-4015. 

\bibitem{DUZ}A. M. Davie, M. Urba\'nski e A. Zdunik, Maximizing measures on metrizable non- compact
spaces. \emph{Proceedings of the Edinburgh Mathematical Society} \textbf{50}(2007),123-151.

\bibitem{RV}R. Freire Jr, V. Vargas, Equilibrium states and zero
temperature limit on topologically transitive countable Markov
shifts. preprint (2015), arxiv: 1511.01527.

\bibitem{GL}E. Garibaldi and A.O. Lopes, On the aubry-mather theory for symbolic dynamics, \emph{Ergodic Theory and Dynamical Systems}\textbf{28}(2008)791-815.

\bibitem{Iommi}
G. Iommi, Ergodic optimization for renewal type shifts, \emph{Monatshefte f\"ur Mathematik} \textbf{150} (2007), 91-95.

\bibitem{Jenkinson}
O. Jenkinson, Ergodic optimization, \emph{Discrete and Continuous Dynamical Systems, Series A}
\textbf{15} (2006), 197-224.

\bibitem{JMU0}
O. Jenkinson, R. D. Mauldin and M. Urba\'nski, Zero temperature limits of Gibbs-equilibrium states for countable
alphabet subshifts of finite type, \emph{Journal of Statistical Physics} \textbf{119} (2005), 765-776.

\bibitem{JMU1}
O. Jenkinson, R. D. Mauldin and M. Urba\'nski, Ergodic optimization for countable alphabet subshifts of finite type,
\emph{Ergodic Theory and Dynamical Systems} \textbf{26} (2006), 1791-1803.

\bibitem{JMU2}
O. Jenkinson, R. D. Mauldin and M. Urba\'nski, Ergodic optimization for non-compact dynamical systems,
\emph{Dynamical Systems} \textbf{22} (2007), 379-388.

\bibitem{Kempton} T. Kempton. Zero temperature limits of Gibbs equilibrium states for countable Markov shifts. \emph{Journal of Statistical Physics}. \textbf{143}(2011), 795-806.

 \bibitem{Kempton02}T. Kempton, Thermodynamic formalism for symbolic dynamical systems. PhD Thesis, The University of Warwick, (2011).

\bibitem{Leplaideur} R. Leplaideur, A dynamical proof for the convergence of Gibbs measures at temperature zero. \emph{Nonlinearity} \textbf{18}(2005) 2847-2880.

\bibitem{LM} A. O. Lopes and J. K. Mengue,  Zeta measures and Thermodynamic Formalism for temperature zero, \emph{Bulletin of the Brazilian Mathematical Society} \textbf{41} (2010), 449-480.

\bibitem{LM1} A.O. Lopes and J. K. Mengue, Selection of measure and a Large Deviation Principle for the general one-dimensional XY model, \emph{Dynamical Systems}, \textbf{29} (2014), 24-39.

\bibitem{Mane} R. Ma\~{n}\'{e}, Generic properties and problems of minimizing measures of
Lagrangian systems, \emph{nonlinearity}, \textbf{9} (1996), 273?310.


\bibitem{MU} R. D. Mauldin and M. Urba\'nski, Gibbs states on the symbolic space
over an infinite alphabet, \emph{Israel J. Math} \textbf{125} (2001) 93-130 .

\bibitem{MU2} R. D. Mauldin and M. Urba\'nski, Graph directed Markov systems: geometry and dynamics of limit sets, Cambridge University Press(2003).

\bibitem{Mengue}J. K. Mengue. Zeta medidas e princ\'ipio dos grandes desvios. PhD Thesis, Universidade Federal do Rio Grande do Sul, (2010).

\bibitem{Mengue2} J. K. Mengue,  Large deviations for equilibrium measures and selection of subaction, \emph{preprint} (2016), arXiv:1608.05881 [math.DS].

\bibitem{Morris} I. D. Morris, Entropy for zero-temperature limits of Gibbs-equilibrium states for countable-alphabet subshifts of finite type. \emph{Journal of Statistical Physics}. \textbf{126} (2007), 315-324.

\bibitem{Pesin} Y. Pesin, On the work of Sarig on countable Markov chains and thermodynamic formalism, \emph{Journal of Modern Dynamics} \textbf{8} (1) (2014), 1-14.

\bibitem{PP} W. Parry and M. Pollicott, Zeta functions and the periodic orbit structure of hyperbolic
dynamics, \emph{Ast\'erisque} (1990), 187-188.

\bibitem{Pe} E. P\'erez, Princ\'ipio dos Grandes Desvios para Estados de Gibbs-Equil\'ibrio sobre Shifts Enumer\'aveis \`a Temperatura Zero. PhD Thesis, University of S\~ao Paulo, (2015).
     
     
\bibitem{AQuas}A. Quas and J. Siefken. Ergodic Optimization 
of super-continuous function. \emph{Ergodic Theory and Dynamical Systems}.
32:2071-2082,(2012).

\bibitem{Ru} D. Ruelle, Thermodynamic Formalism: The Mathematical Structure of Equilibrium Statistical Mechanics. 2nd Edition, Cambridge University Press, (2004).

\bibitem{Sarig01} O. Sarig, Thermodynamic formalism for countable Markov shifts,\emph{Ergodic Theory and Dynamical Systems}. \textbf{19}(6)(1999), 1565-1593.

\bibitem{Sarig02} O. Sarig. Existence of Gibbs measures for countable Markov shifts, \emph{Proc. of AMS}. \textbf{131}(6)(2003), 1751-1758.

\bibitem{Sarig03}O. Sarig. Lecture Notes on Thermodynamic Formalism for Topological Markov Shifts. Penn State, (2009).

\bibitem{Sarig04} O. Sarig. Thermodynamic formalism for countable Markov shifts. \emph{Proceedings of Symposia in Pure Mathematics} \textbf{89} (2015), 81-117.

\bibitem{Sarig05} O. Sarig. Phase Transitions for Countable Topological Markov Shifts. \emph{Communications in Mathematical Physics} \textbf{217}(2001), 555-577.

\bibitem{Sarig06} O. Sarig. On an example with topological pressure which is not analytic. \emph{C. R. Acad. Sci. Paris} serie I, \textbf{330}  (2000), 311-315.

\bibitem{Sinai} Y.G. Sinai. Gibbs measures in ergodic theory. \emph{Russian Mathematical Surveys} \textbf{27}, no. 4, (1972), 21-64.



\end{thebibliography}
\end{document}